\documentclass[12pt]{article}
\usepackage[english]{babel}
\usepackage[utf8]{inputenc}
\usepackage[T1]{fontenc}
\usepackage{tikz}
\usepackage{graphicx}
	\graphicspath{{images/}}
\usepackage{url}
\usepackage{enumitem}
\usepackage{authblk}

\usepackage{makeidx}
\makeindex
\usepackage{amsmath,amsfonts,amsthm,amssymb}
\usepackage{typearea}
\usepackage{graphicx}
\usepackage{geometry}
\newcommand{\Esp}{\mathbb{E}}

\theoremstyle{definition}
\newtheorem{definition}{Definition}
\theoremstyle{remark}

\newtheoremstyle{mytheorem}{0.5cm}{0.2cm}{\slshape}{ }{\bfseries}{.}{ }{}
\theoremstyle{mytheorem}

\newtheorem{theorem}[definition]{Theorem}

\newtheorem{lemma}[definition]{Lemma}

\renewcommand{\P}{\mathbf{P}}

\DeclareMathOperator{\var}{{Var}}

\DeclareMathOperator{\cov}{cov}
\DeclareMathOperator{\PAI}{PAI}
\DeclareMathOperator{\TLLN}{TLLN}
\DeclareMathOperator{\Cset}{\mathcal{C}}
\DeclareMathOperator{\one}{{ 1\hspace*{-0.55ex}I}}

\renewcommand{\epsilon}{\varepsilon}

\renewcommand{\phi}{\varphi}

\newcommand*\diff{\mathop{}\!\mathrm{d}}

\newlength{\querylen}
\usepackage{fancybox}

\title{
Replica-Mean-Field Limits of\\
Fragmentation-Interaction-Aggregation Processes\\
}
\date{}
\author{Fran\c{c}ois Baccelli \thanks{INRIA, Paris, France and D\'epartement d'informatique de l'ENS, ENS, CNRS, PSL University, Paris, France} \and Michel Davydov \thanks{INRIA, Paris, France and D\'epartement d'informatique de l'ENS, ENS, CNRS, PSL University, Paris, France}  \and Thibaud Taillefumier \thanks{Department of Mathematics and Department of Neuroscience, University of Texas, Austin, TX}}
\begin{document}

\maketitle

\begin{abstract}
Network dynamics with point-process-based interactions are of paramount modeling interest.
Unfortunately, most relevant dynamics involve complex graphs of interactions for which an exact computational treatment is impossible. 
To circumvent this difficulty, the replica-mean-field approach focuses on randomly interacting replicas of the networks of interest.
In the limit of an infinite number of replicas, these networks become analytically tractable under the so-called ``Poisson Hypothesis''.
However, in most applications, this hypothesis is only conjectured.
Here, we establish the Poisson Hypothesis for a general class of discrete-time, point-process-based dynamics, that we propose to call fragmentation-interaction-aggregation processes, and which are introduced in the present paper.
These processes feature a network of nodes, each endowed with a state governing their random activation. 
Each activation triggers the fragmentation of the activated node state and the transmission of interaction signals to downstream nodes.
In turn, the signals received by nodes are aggregated to their state.
Our main contribution is a proof of the Poisson Hypothesis for the replica-mean-field version of any network in this class.  
The proof is obtained by establishing the propagation of asymptotic independence for state variables in the limit of an infinite number of replicas.
Discrete time Galves-L\"{o}cherbach neural networks are used as a basic instance and illustration of our analysis.
\end{abstract}

\paragraph{keywords}{point process; Hawkes process; Markov process; mean-field theory; replica model; Poisson hypothesis; neural network; queuing network}
\paragraph{ams}{60K35}{60G55}

\newpage
\section{Introduction}

Epidemics propagation, chemical reactions, opinion dynamics, flow control in the Internet, and even neural computations can all be modelled via punctuate interactions between interconnected agents \cite{Pastor_Satorras_2015}\cite{Gupta_2014}\cite{Amblard_2004}\cite{baccelli:2002}\cite{Shriki_2013}.
The phenomena of interest in this context are idealized as network dynamics on a graph of agents which interact via point processes: edges between agents are the support of interactions, with edge-specific point processes registering the times at which these interactions are exerted.
Such point-process-based network dynamics constitute a very versatile class of models able to capture phenomena in natural sciences, engineering, social sciences and economics.
%
However, this versatility comes at the cost of tractability as the mathematical analysis of these dynamics is impossible except for the simplest network architectures. 
As a result, one has to resort to simplifying assumptions to go beyond numerical simulations.

Generic point-process-based networks are computationally untractable because their stochastic dynamics does not appear to belong to any known parametric class of point processes.
Replica mean-field (RMF) limits are precisely meant to circumvent this obstacle \cite{Baccelli_2019}. 
The RMF limit of a given network is an extension of this network built in such a way that
interaction point processes are parametric, e.g., Poisson. This extended network is made of infinitely many replicas of the initial network, all with the same basic structure, but with randomized interactions across replicas. 
The interest in RMF limits stems from the fact that they offer
tractable version of the original dynamics that retain some of its most important features. 
The fact that Poisson point processes arise in the
RMF version of a network is called the Poisson Hypothesis.
Thus formulated, the Poisson Hypothesis originates from communication network theory \cite{kleinrock2007communication} and is distinct from replica approaches developed in statistical physics \cite{Castellani_2005}.

Although intuitively clear and despite its usefulness, the Poisson Hypothesis is often only conjectured and/or numerically validated.
The purpose of this work is to rigorously establish the Poisson Hypothesis for the RMF limits of a broad class of point-process-based network dynamics in discrete time introduced in the present paper. This class, which will be referred to as  fragmentation-interaction-aggregation processes (FIAPs) below, includes important classes of queuing networks as special cases, as well as discrete time Galves-L\"{o}cherbach (GL) neural networks.

Galves-L\"{o}cherbach networks can be viewed as coupled Hawkes processes with spike-triggered memory resets.
Because of these memory resets, it can be shown that the dynamics of finite-size Galves-L\"{o}cherbach networks is Markovian \cite{robert_touboul}.
The RMF limit of the GL case was studied from a computational standpoint in \cite{Baccelli_2019} but in continuous time.
In the next paragraph, we use results established in \cite{Baccelli_2019} to illustrate how the Poisson Hypothesis yields tractable mean-field equations for the stationary dynamics of these RMF limits.
This is done for the simplest example of GL networks, called the ``counting-neuron'' model.

\subsection*{Illustration from the study of spiking neural networks}

The counting-neuron model consists of a fully-connected network of $K$
exchangeable neurons with homogeneous synaptic weights $\mu$. 
For each neuron $i$, $1 \leq i \leq K$, the continuously time-indexed stochastic intensity $\lambda_i$ increases by $\mu>0$
upon reception of a spike and resets upon spiking to its base rate $b>0$.
Thus, its stochastic intensity is $\lambda_i(t)=b+\mu C_i(t)$, where $C_i(t)$ is
the number of spikes received at time $t$ since the last reset. 
It can be shown that the network state $\lbrace C_1(t), \dots , C_K(t) \rbrace$ has a well-defined stationary distribution.
Despite of the simplicity of the model, analytic characterizations of the stationary state, including the stationary spiking rate, 
are hindered by the fact that the law of the point process of spike receptions is not known.

To circumvent this hindrance, the RMF setting proposes to compute stationary spiking rates in 
infinite networks that are closely related to the original finite-size networks.
Informally, the counting-model RMF is constructed as follows:
for a $K$-neuron counting model and for an integer $M>0$, the $M$-replica model consists of $M$ replicas,
each comprising $K$ counting neurons. 
Upon spiking, a neuron $i$ in replica $m$, indexed by $(m,i)$, delivers 
spikes with synaptic weight $\mu$ to the $K-1$ neurons $(v_j,j)$, $j\neq i$, where the replica destination $v_j$ is chosen uniformly at random for all $j$. 
RMF networks are defined in the limit of an infinite number of replicas, namely infinite $M$ but fixed and finite $K$.
The Poisson Hypothesis then states that the dynamics of replicas become asymptotically independent in the limit $M \to \infty$, 
and that each neuron receives spikes from independent Poisson point processes. It is shown in \cite{Baccelli_2019} that
as a consequence of this Poisson property, 
%
the stationary state is characterized by
a single ODE bearing on $G$, the probability-generating function (PGF) of a neuron count $C$:
\begin{eqnarray}\label{eq:intro}
 \beta \!- \!\mu z G'(z) \!+\! \big(\beta (K\!-\!1)(z\!-\!1)\!-\!b \big)G(z) \!=\!0    \, .
\end{eqnarray} 
The simplifications warranted by the Poisson Hypothesis in the above ODE characterization comes at the cost of introducing the spiking rate $\beta$
as an unknown parameter in \eqref{eq:intro}. 
As the ODE \eqref{eq:intro} is otherwise analytically tractable, characterizing the RMF stationary state amounts to specifying the unknown firing rate $\beta$. 
Then, the challenge of the RMF approach consists in specifying the unknown firing rate via purely analytical considerations about a parametric system of ODEs.
It turns out that requiring that the solution of \eqref{eq:intro} be analytic, as any PGF shall be, is generally enough to exhibit 
self-consistent relations about the stationary rates.
For instance, the RMF stationary spiking rate $\beta$ of the RMF counting model is shown to be determined as the unique solution of  
\begin{eqnarray}\label{eq:selfCons}
\beta = \frac{\mu c^a e^{-c}}{\gamma (a,c)} \quad \mathrm{with} \quad a = \frac{(K-1)\beta+b}{\mu} \quad \mathrm{and} \quad c = \frac{(K-1)\beta}{\mu} \, ,
\end{eqnarray}
where $\gamma$ denotes the lower incomplete Euler Gamma function.

We have shown that the above approach generalizes to networks with continuous state space, heterogeneous network structures \cite{Baccelli_2019}, and including pairwise correlations \cite{Baccelli_2020}. 
In all cases, the Poisson Hypothesis is the cornerstone of a computational treatment.
As a key step toward establishing the Poisson Hypothesis for continuous-time network dynamics, the goal of this work is to prove it for 
a broad class of discrete-time dynamics, which we refer to as fragmentation-interaction-aggregation processes.

\subsection*{Fragmentation-interaction-aggregation processes}

In FIAPs, agents are graph nodes endowed with a state that evolves over time. 
The nodes are coupled via point processes which model punctuate interactions.
Specifically, each node's state evolves in response to its input point process, and generates an output point process in a state-dependent manner.
In all generality, the transformation of input into output point process can be viewed as a random map. 
In FIAPs, this map is defined through the following dynamics:
$(i)$ The \emph{fragmentation} process is triggered by local {\em activation} events
taking place on each node and which occur with
a probability that depends on the state of the node.
$(ii)$ Each fragmentation event in turn triggers \emph{interactions} between
the nodes by creating input events in the neighboring nodes.
$(iii)$ Finally, the \emph{aggregation} process consists in the integration of the input point processes to the states of each node.

Thus broadly defined, FIAPs offer a simple albeit general framework to analyze the phenomena alluded to above.
The precise definition of FIAPs is given as follows:
\begin{definition}
\label{def_FIAP}
An instance of the class $\Cset$ of discrete
\textit{fragmentation-interaction-aggregation processes} 
is determined by:
\begin{itemize}
\item An integer $K$ representing the number of nodes;
\vspace{-5pt}
\item A collection of initial conditions for the integer-valued state variables at step zero, which we denote by $\{X_{i}\}$, where $i \in \{1,\ldots,K\}$; 
\vspace{-5pt}
\item A collection of fragmentation random variables $\{U_{i}\}$, which are i.i.d. uniform in $[0,1]$ and independent from $\{X_{i}\}$, where $i \in \{1,\ldots,K\}$; 
\vspace{-5pt}
\item A collection of {\em fragmentation functions} $\{g_{1,i} : \mathbb{N} \rightarrow \mathbb{N}\}_{i \in \{1,\ldots,K\}}$ and $\{g_{2,i} : \mathbb{N} \rightarrow \mathbb{N}\}_{i \in \{1,\ldots,K\}}$; 
\vspace{-5pt}
\item A collection of bounded {\em interaction functions} $\{h_{i,j} : \mathbb{N} \rightarrow \mathbb{N}\}_{i,j \in \{1,\ldots,K\}}$; 
\vspace{-5pt}
\item A collection of {\em activation probabilities} $\{\sigma_i(0), \sigma_i(1),\ldots\}_{i\in \{1,\ldots,K\}}$ verifying the 
conditions $\sigma_i(0)=0$,  and
$0<\sigma_i(1)\le \sigma_i(2)\le \cdots\le 1$ for all $i$.
\end{itemize}
The associated dynamics take as input 
the initial integer-valued state variables $\{X_{i}\}$
and define the state variables at the next step as
\begin{equation}
\label{eq_gendef_1}
Y_{i}=g_{1,i}(X_{i})\one_{\{U_{i}<\sigma_i(X_{i})\}}+g_{2,i}(X_{i})\one_{\{U_{i}>\sigma_i(X_{i})\}}+A_{i}, \quad \forall i=1,\ldots,K,
\end{equation}
with arrival processes 
\begin{equation}
\label{eq_gendef_2}
A_{i}=\sum_{j \neq i} h_{i,j}(X_{j})\one_{\{U_{j}<\sigma_j(X_{j})\}},
\quad \forall i=1,\ldots,K.
\end{equation}
\end{definition}

The interpretation is as follows: node $i$ activates with
probability $\sigma_i(k)$ if its state $X_i$ is equal to $k$.
The state of this node is fragmented to $g_{1,i}(k)$ upon activation and to $g_{2,i}(k)$ otherwise.
This activation triggers an input of $h_{j,i}(k)$ units to node $j$.
Hence, the interaction functions encode the structure of the graph.
The variable $A_i$ gives the total number of arrivals to
node $i$. This variable is aggregated to the state of the
node as seen in \eqref{eq_gendef_1}.
Note that considering $\sigma_i(0)=0$ for all $i$ ensures that state variables in state 0 cannot be fragmented.

The FIAP class $\Cset$ encompasses many network dynamics relevant to queuing theory and mathematical biology. 
For example, taking $g_{1,i}(k)=k-1, g_{2,i}(k)=k$ and $h_{i,j}(k)=\one_{\{i=j+1 \mod K\}}$, we recover an instance 
of Gordon-Newell queuing networks \cite{Kleinrock_2}. 
Taking $g_{1,i}(k)=0, g_{2,i}(k)=k$ and $h_{i,j}(k)=\mu_{i,j}\in \mathbb{N}$ defines a discrete instance of Galves-L\"{o}cherbach dynamics for neural 
networks introduced above.
Taking $g_{1,i}(k)=\lfloor \frac k 2 \rfloor$ and $g_{2,i}(k)=k+1$, corresponds to aggregation-fragmentation processes modelling, e.g., 
TCP communication networks \cite{baccelli:2002}.
The class $\Cset$ also includes certain discrete time Hawkes processes. Namely, if for each coordinate of a Hawkes vector process, 
we define its state as the sum over time of all its variations, then all discrete Hawkes processes that are Markov with respect to their 
so-defined state are in $\Cset$. 
Thus, the results of the present paper have potential computational implications in a wide set of application domains 
beyond the neural network setting used above to illustrate them. 

The present paper is focused on discrete time versions of this type of dynamics as in, e.g., \cite{seol2015limit} \cite{Cessac_2007}; note that 
continuous instances were also considered in the literature such as in \cite{Delattre}, \cite{Masi2015HydrodynamicLF}.

\subsection*{Replica models for fragmentation-interaction-aggregation networks}

Finite RMF models are defined as a coupling of replicas of the network of interest by randomized routing decisions.
For a FIAP, the state of its $M$-replica model is thus specified by a collection of state variables $X^{M}_{m,i}$, where $m$ is the index of the replica and $i$ corresponds to the index of the node in the original network.
Instead of interacting with nodes within the same replica, an activated node $i$ in replica $m$ interacts with a downstream node $j$ from a replica $n$ chosen uniformly at random and independently.
This randomization preserves essential features of the original dynamics such as the magnitude of interactions between nodes but degrades the dependence structure between nodes.
Indeed, over a finite period of time, the probability for a particular node to receive
an activation from another given node scales as $1/M$.
Thus, as the number of replicas increases, interactions between distinct replicas become ever scarcer, intuitively leading to replica independence when $M \to \infty$. 
This asymptotic independence is the root of RMF computational tractability.


Here is the precise definition of the finite-replica version of a FIAP:

\begin{definition}
For any process in $\Cset$, the associated $M$-replica dynamics is entirely specified by 
\begin{itemize}
\item A collection of initial conditions for the integer-valued state variables at step zero, which we denote by $\{X^{M}_{n,i}\}$, where $n \in \{1,\ldots,M\}$ and $i \in \{1,\ldots,K\}$, such that for all $M, n$ and $i$, $X^M_{n,i}=X_i$; 
\vspace{-5pt}
\item A collection of fragmentation random variables $\{U_{n,i}\}$, which are i.i.d. uniform in $[0,1]$ and independent from $\{X^{M}_{n,i}\}$, where $n \in \{1,\ldots,M\}$ and $i \in \{1,\ldots,K\}$; 
\vspace{-5pt}
\item A collection of i.i.d. \textit{routing} random variables $\{R^M_{m,j,i}\}$ independent from $\{X^{M}_{n,i}\}$ and $\{U_{n,i}\}$, uniformly distributed on  $\{1,\ldots,M\}\setminus\{m\}$ for all $i,j \in \{1,\ldots,K\}$ and $m \in \{1,\ldots,M\}$. In other words, if $R^M_{m,j,i}=n$, then an eventual activation of node $j$ in replica $m$ at step 0 induces an arrival of size $h_{i,j}(X^M_{m,j})$ in node $i$ of replica $n$, and $n$ is chosen uniformly among replicas and independently from the state variables. Note that these variables are defined regardless of the fact that an activation actually occurs.
Also note that for $i'\ne i$, the activation in question will 
induce an arrival in node $i'$ of replica $n'$, with $n'$
sampled in the same way but independently of $n$.
\end{itemize}
Then, the integer-valued state variables at step one, denoted by $\{Y^{M}_{n,i}\}$, are given by the $M$-RMF equations
\begin{equation}
\label{eq_gen_rmf_1}
Y^{M}_{n,i}=g_{1,i}(X^{M}_{n,i})\one_{\{U_{n,i}<\sigma_i(X^{M}_{n,i})\}}+g_{2,i}(X^{M}_{n,i})\one_{\{U_{n,i}>\sigma_i(X^{M}_{n,i})\}}+A^M_{n,i},
\end{equation}
where $g_{1,i}$, $g_{2,i}$ denotes fragmentation functions, $\sigma_i$ denotes activation probabilities, and where
\begin{equation}
\label{eq_gen_rmf_2}
A^M_{n,i}=\sum_{m \neq n}\sum_{j \neq i} h_{i,j}(X^M_{m,j})\one_{\{U_{m,j}<\sigma_i(X^{M}_{m,j})\}}\one_{\{R^M_{m,j,i}=n\}}
\end{equation}
is the number of arrivals to node $i$ of replica $n$ via the interaction functions $h_{i,j}$. 
\end{definition}


RMF models are only expected to become tractable when individual replicas become independent.
This happens in the limit of an infinite number of replicas, i.e., in the so-called {\em RMF limit} \cite{Baccelli_2019}.
In this RMF limit, asymptotic independence between replicas follows from the more specific Poisson Hypothesis. 
The Poisson Hypothesis states that spiking deliveries to distinct replicas shall be asymptotically distributed as independent Poisson (or compound) point processes.
Such a hypothesis, which has been numerically validated for certain RMF networks, has been conjectured for linear Galves-L\"{o}cherbach dynamics in \cite{Baccelli_2019}.
Proving the validity of the Poisson Hypothesis for the RMF limits of the much more general FIAPs is the purpose of the present work.

\subsection*{Methodology for proving the Poisson Hypothesis}
Classical mean-field approximations of a given network are obtained by considering the limit of the original network when a certain characteristic of the network – typically the number of nodes – goes to infinity. When the dynamics of the nodes are synchronous, one gets a discrete time dynamical system. The term mean-field comes from the fact that in such network limits, the effect that individual nodes have on one another are approximated by a single averaged effect, typically an empirical mean. In the limit, this empirical mean usually converges to an expectation term through a \textit{propagation of chaos} result \cite{Szn:89} which leads to analytical tractability. In replica mean-fields, there is no such empirical mean over the nodes of the network; the mean-field simplification comes from the random routing operations between replicas. The input point process in the $M$-replica model consists in a superposition of $M$ rare point processes, which informally explains why Poisson (or compound Poisson) processes arise at the limit.
For classical mean-fields, different techniques have been developed to prove the existence and the convergence to the mean-field limit. 
Standard techniques include the use of the theory of nonlinear Markov processes \cite{Vladimirov_2018} and stochastic approximation algorithms\cite{Benaim_LeBoudec} for continuous time dynamics, and induction techniques
which assume the existence of limits at time zero and extend the result by induction \cite{LeBoudec07} for dynamics in discrete time. Refinements to the latter approach can be made in order to obtain explicit rates of convergence \cite{Gast_2018}.
The approach developed for the RMF case belongs in spirit to the third class of techniques. We suppose that the property of asymptotic independence (see Definition \ref{def_pai_1}) holds for the state variables at time zero. We then prove that this property is preserved by the dynamics of the $M$-replica model and thus holds by induction for any finite time. Thus, we focus hereafter on the one-step transition of the model from time 0 to time 1. We show that this asymptotic independence hypothesis implies both convergence in distribution and an ergodic type property that we call the \textit{triangular law of large numbers}. We apply this law of large numbers to the input process to a single node to show that Poisson (or compound Poisson) processes appear in the replica mean-field limit indeed. Let us stress that this proof is by induction. The fact that the main difficulty 
consists in proving the induction step should not hide the fact that the result relies in crucial way on the assumption that at step 0, the initial
state variables satisfy the asymptotic independence property. Whether the result can be extended to more general initial conditions 
is an open question at this stage.

\subsection*{Structure of the Paper}

For the sake of clarity in exposition, we start with the proof of the Poisson Hypothesis for the special case of neural networks
first before extending it to general FIAPs.
More precisely, we first consider the symmetric neural network case, which is a fully symmetric Galves-L\"{o}cherbach model \cite{Galves_2013} in discrete time. We introduce the model in Section \ref{sec:sglm} and prove the Poisson Hypothesis 
in Section \ref{sec:thepr}.
We then extend the proof to the class of FIAPs defined above. 
We first consider
the symmetric case in Section \ref{sec:sfiap} and then the general case in Section \ref{sec:gfiap}.
Finally, some extensions are discussed in Section \ref{sec:ext}.

\section{The symmetric Galves-L\"{o}cherbach model}
\label{sec:sglm}
\subsection{The symmetric RMF network model}
We consider a network of $K$ spiking neurons. We suppose that the behavior of each neuron is determined by a random variable representing the membrane potential of the neuron. Each neuron spikes at a rate depending on its state variable. 
Let $X=\{X_{i}\}$ be the integer-valued state variables at step 0, where  $i \in \{1,\ldots,K\}.$ Let $Y=\{Y_{i}\}$ be the integer-valued state variables at time one. The system continues to evolve in discrete time with all corresponding state variables defined by induction.

Let $\sigma : \mathbb{N} \rightarrow [0,1]$ be the spiking probabilities of the neurons. Namely, $\sigma(k)$ is the probability that a neuron in state $k$ spikes. We consider that $\sigma(0)=0$, accounting for the fact that a neuron in state 0 never spikes. We also consider that $\sigma(1)>0$ and that $\sigma$ is non-decreasing. Let $\{U_{i}\}$ be uniformly distributed i.i.d. random variables independent from $\{X_{i}\}$.  We then write the following evolution equation for the state of the system:
\begin{equation}
\label{eq_def1}
Y_{i}=\one_{\{U_{i}>\sigma(X_{i})\}} X_{i}+A_{i},
\end{equation}
where
\begin{equation}
\label{eq_def2}
A_{i}=\sum_{j \neq i} \one_{\{U_{j}<\sigma(X_{j})\}}
\end{equation}
is the number of arrivals to neuron $i$.

Here, the fragmentation is complete if $U_{i}<\sigma(X_{i})$, 
namely if there is a spike, in which case the state variable 
is reset (jumps to 0). Otherwise there is no fragmentation at all and the state variable is left unchanged. In both cases, the arrivals $A_i$ are aggregated to the state. 

The RMF model described below is a discrete time version of the model introduced in \cite{Baccelli_2019}. Namely, we consider a collection of $M$ identically distributed replicas of the initial set of $K$ neurons. Let $X=\{X^{M}_{n,i}\}$ be the integer-valued state variables at step 0, where  $n \in \{1,\ldots,M\}, i \in \{1,\ldots,K\}.$ Let $Y=\{Y^{M}_{n,i}\}$ be the integer-valued state variables at time one. Let $U=\{U_{n,i}\}$ be uniformly i.i.d. random variables in $[0,1]$ independent from $\{X^{M}_{n,i}\}$. 
Let $R=\{R^M_{m,i,j}\}$ be i.i.d. \textit{routing} random variables independent from $\{X^{M}_{n,i}\}$ and $\{U_{n,i}\}$, uniformly distributed on  $\{1,\ldots,M\}\setminus\{m\}$ for all $i,j \in \{1,\ldots,K\}$ and $m \in \{1,\ldots,M\}$. 
The replica model has the following evolution equation:

\begin{equation}
\label{eq_rmf1}
Y^{M}_{n,i}=\one_{\{U_{n,i}>\sigma(X^{M}_{n,i})\}} X^{M}_{n,i}+A^M_{n,i},
\end{equation}
where
\begin{equation}
\label{eq_rmf2}
A^M_{n,i}=\sum_{m \neq n}\sum_{j \neq i} \one_{\{U_{m,j}<\sigma(X^{M}_{m,j})\}}\one_{\{R^M_{m,j,i}=n\}}
\end{equation}
is the number of arrivals to neuron $i$ of replica $n$.


\subsection{Pairwise asymptotic independence and consequences}

Our goal is to show the propagation of chaos and the Poisson Hypothesis in this system. In other words, we want to show that the arrivals to two distinct replicas are asymptotically independent and the number of arrivals to one replica is asymptotically Poisson distributed. We begin by considering the fully exchangeable case with equal weights, but we will consider the general case later. In order to do so, we choose to characterize the propagation of chaos through the following properties:

\begin{definition}
\label{def_pai_1}
Given $M \in \mathbb{N}$, given an array of integer-valued random variables $Z=\{Z^{M}_{n,i}\}_{1 \leq n \leq M, 1 \leq i \leq K}$ such that for all fixed $M$, the random variables $Z^M_{n,i}$ are exchangeable in $n$ and $i$, we say that the variables $Z^M_{n,i}$ are \textit{pairwise asymptotically independent}, which we will denote $\PAI(Z)$, if there exists an integer-valued random variable $\tilde{Z}$ such that for all $(n,i) \neq (m,j)$, for all $u,v \in [0,1]$,
\begin{equation}
\label{eq_pai1}
\lim_{N \rightarrow \infty}\Esp[u^{Z^N_{n,i}}v^{Z^N_{m,j}}] =\Esp[u^{\tilde{Z}}]\Esp[v^{\tilde{Z}}].
\end{equation}

\end{definition}

\begin{definition}
\label{def_tlln_1}
Given $M \in \mathbb{N}$, given an array of integer-valued random variables $Z=\{Z^{M}_{n}\}_{n \in \{1,\ldots,M\}}$ such that for all fixed $M$, the random variables $Z^M_{n}$ are exchangeable in $n$, we say that $Z$ verifies the \textit{triangular law of large numbers}, denoted by $\TLLN(Z)$, if there exists an integer-valued random variable $\tilde{Z}$ such that for all functions $f:\mathbb{N}\rightarrow \mathbb{R}$ with compact support, we have the following limit in $L^2$:

\begin{equation}
\label{eq_tlln1}
\lim_{N\rightarrow\infty}\frac{1}{N}\sum_{n=1}^N f(Z_{n}^N)=\Esp[f(\tilde{Z})].
\end{equation} 

\end{definition}

Here are a few remarks about these definitions. First, note that if an array of random variables $Z$ satisfies $\PAI(Z)$, then for all $n$ and $i$, $Z^M_{n,i}$ converges in distribution to $\tilde{Z}$ as $M \rightarrow \infty$. This can be seen by taking $v=1$ in the definition. 
By considering the case where $Z^M_{n}=Z^1_1$ for all $n$ and $M$, we see that the convergence in distribution of $Z^M_{n}$ does not imply $\TLLN(Z)$. However, we show below that for all arrays of random variables $Z=\{Z^{M}_{n,i}\}_{n \in \{1,\ldots,M\}, i \in \{1,\ldots,K\}}$ satisfying $\PAI(Z)$, for all $i$, $Z_i=\{Z^{M}_{n,i}\}_{n \in \{1,\ldots,M\}}$ satisfies $\TLLN(Z_i)$. In other words, pairwise asymptotic independence of an array of random variables implies that these random variables verify the triangular law of large numbers. Finally, note that an array of integer-valued random variables $Z$ satisfies $\PAI(Z)$ iff the random variables are asymptotically independent in the sense that for all $(n,i) \neq (m,j)$
\begin{equation}
\label{eq_ai_equiv}
\P(Z^M_{n,i} \in B_1, Z^M_{m,j} \in B_2) \rightarrow \P(\tilde{Z} \in B_1)\P(\tilde{Z} \in B_1)
\end{equation}
when $M \rightarrow \infty$ for $B_1, B_2 \in \mathcal{B}(\mathbb{R})$.

The following characterization of $L^2$ convergence will be used throughout this paper:
\begin{lemma}
\label{lem_L2}
Let $(X_n)$ be random variables with finite second moments. Then there exists a constant $c$ such that $X_n \rightarrow c$ in $L^2$ when $n \rightarrow \infty $ iff
\begin{enumerate}
\item $\Esp[X_n] \rightarrow c$ when $n \rightarrow \infty$
\item $\var(X_n) \rightarrow 0$ when $n \rightarrow \infty.$
\end{enumerate}
\end{lemma}

This follows directly from the definition of $L^2$ convergence.

The following lemma describes the relation between pairwise asymptotic independence and the triangular law of large numbers.

\begin{lemma}
\label{lem_pai_tlln}
Let $M \in \mathbb{N}$, let $Z=\{Z^{M}_{n,i}\}_{n \in \{1,\ldots,M\}, i \in \{1,\ldots,K\}}$ be an array of integer valued random variables verifying $\PAI(Z)$. Then, for all $i$, $Z_i=\{Z^{M}_{n,i}\}_{n \in \{1,\ldots,M\}}$ satisfies $\TLLN(Z_i)$. 
\end{lemma}
\begin{proof}

Let $f : \mathbb{N} \rightarrow \mathbb{R}$ be a function with compact support. We use Lemma \ref{lem_L2}. We fix $i \in \{1,\ldots,K\}$ that we omit in the rest of the proof. We have 
\begin{equation*}
\begin{split}
\var\left(\frac{1}{M}\sum_{n=1}^M f(Z^M_n)\right)
&=\frac{1}{M^2}\left(\sum_{n=1}^M \var\left(f(Z^M_n)\right)+\sum_{p \neq q}\cov[f(Z^M_p),f(Z^M_q)]\right) \\
&=\frac{1}{M}\var\left(f(Z^M_1)\right)+\frac{M(M-1)}{M^2}\cov[f(Z^M_1),f(Z^M_2)],
\end{split}
\end{equation*}
the last equality holding by exchangeability between replicas.
Both terms on the right hand side go to 0 when $M \rightarrow \infty$. For the first term, this follows from the boundedness of $f$. For the second, we first show the result for indicator functions. Let $B \in \mathcal{B}(\mathbb{R})$ and let $f$ be defined by $f(n)=\one_{\{n\in B\}}$. Then we have
\begin{equation}
\label{eq_cov1}
\cov[f(Z^M_1),f(Z^M_2)]=\P(Z^M_1\in B,Z^M_2\in B)-\P(Z^M_1\in B)\P(Z^M_2\in B),
\end{equation}
which goes to 0 when $M \rightarrow \infty$ by $\PAI(Z)$.
This immediately extends to functions with compact support since they only take a finite number of values.
Moreover, for all such functions, $\Esp[\frac{1}{M}\sum_{n=1}^M f(Z^M_n)] \rightarrow \Esp[f(\tilde{Z})]$ when $M \rightarrow \infty$ as a direct consequence of the fact that for integer-valued random variables, convergence in distribution of $Z^M$ to $\tilde{Z}$ is equivalent to the convergence $\P(Z^M=k) \rightarrow \P(\tilde{Z}=k)$ for all $k \in \mathbb{N}$. This concludes the proof.
\end{proof}

For our subsequent needs, we also establish the following result: we show that pairwise asymptotic independence implies a property that is slightly more general than the triangular law of large numbers, where we allow the function $f$ to depend on an array of i.i.d. random variables $U=\{U_{n,i}\}_{n \in \{1,\ldots,M\}, i \in \{1,\ldots,K\}}$, independent from the rest of the dynamics.

\begin{lemma}
\label{lem_gen_tlln}
Let $M \in \mathbb{N}$, let $Z=\{Z^{M}_{n,i}\}_{n \in \{1,\ldots,M\}, i \in \{1,\ldots,K\}}$ be an array of integer valued random variables verifying $\PAI(Z)$. Then for all bounded functions $f : \mathbb{N}\times [0,1] \rightarrow \mathbb{R}$ with compact support, for all i.i.d. sequences of random variables $U=\{U_{n,i}\}_{n \in \{1,\ldots,M\}, i \in \{1,\ldots,K\}}$ independent from $Z$, there exists $\tilde{U}$ independent from $\tilde{Z}$ and $Z$ such that, for all $i \in \{1,\ldots,K\}$, we have the following limit in $L^2$:
\begin{equation}
\label{eq_tlln2}
\lim_{M\rightarrow\infty}\frac{1}{M}\sum_{n=1}^M f(Z_{n,i}^M, U_{n,i})=\Esp[f(\tilde{Z},\tilde{U})].
\end{equation}
\end{lemma}

Note that compared to Definition \ref{def_tlln_1}, we consider that the functions are bounded, a condition that was automatically fulfilled for functions with compact support on $\mathbb{N}$.

\begin{proof}
We proceed as in the last lemma, conditioning on the $U_{n,i}$ when necessary.
Let $M \in \mathbb{N}$, let $i \in \{1,\ldots,K\}$. We will omit this index in the rest of the proof.
By exchangeability between replicas, defining $\tilde{U}=U_1$, we have
\begin{equation*}
\Esp\left[\frac{1}{M}\sum_{n=1}^M f(Z^M_{n}, U_n)\right] =\Esp[f(Z^M_1,U_1)] 
=\Esp[f(Z^M_1,\tilde{U})]
\end{equation*}
Since $Z^M_1$ converges in distribution to $\tilde{Z}$ when $M \rightarrow \infty$, and since $Z^M_1$ is integer-valued and $f$ is bounded, for all $u \in [0,1], \Esp[f(Z^M_1,u)] \rightarrow \Esp[f(\tilde{Z},u)]$ when $ M \rightarrow \infty$.Therefore, since $\tilde{U}$ is independent from $Z$ and $\tilde{Z}$,
$\Esp[f(Z^M_1,\tilde{U})] \rightarrow \Esp[f(\tilde{Z},\tilde{U})]$ when $M \rightarrow \infty$ a.s..
Finally, 
\begin{equation}
\label{eq_tlln3}
\Esp\left[\frac{1}{M}\sum_{n=1}^M f(Z^M_{n},U_n)\right] \rightarrow \Esp\left[f(\tilde{Z},\tilde{U})\right]
\end{equation}
when $M \rightarrow \infty.$
Moreover, 
\begin{eqnarray*}
& & \hspace{-1.5cm}\var\left(\frac{1}{M}\sum_{n=1}^M f(Z^M_{n}, U_n)\right) \\
& = & \frac{1}{M^2}\sum_{n=1}^M \var\left(f(Z^M_n,U_n)\right) 
+\frac{1}{M^2}\sum_{n\neq n'}\cov\left[f(Z^M_n,U_n),f(Z^M_{n'},U_{n'})\right]\\
&=&\frac{1}{M}\var\left(f(Z^M_1,U_1)\right)
+\frac{M(M-1)}{M^2}\cov\left[f(Z^M_1,U_1),f(Z^M_2,U_2)\right],
\end{eqnarray*}
the last equality stemming from exchangeability between replicas.
When $M \rightarrow \infty$, the first term goes to 0 because $f$ is bounded. For the second term, since the $\{Z^M_n\}$ and the $\{U_n\}$ are independent and the $\{U_n\}$ are i.i.d., we can proceed as above. Namely, let $B, C \in \mathcal{B}(\mathbb{R})$. Let $f$ be defined by $f(n,t)=\one_{\{n \in B\}}\one_{\{t \in C\}}.$ Then we have
\begin{eqnarray*}
& & \hspace{-1.0cm}
\cov[f(Z^M_1,U_1),f(Z^M_2,U_2)]\\ &=&\P(Z^M_1\in B,Z^M_2\in B, U_1 \in C, U_2 \in C) 
-\P(Z^M_1\in B, U_1 \in C)\P(Z^M_2\in B, U_2 \in C) \\
&=&\left(\P(Z^M_1\in B,Z^M_2\in B)-\P(Z^M_1\in B)\P(Z^M_2\in B)\right)
\P(U_1 \in C)\P(U_2 \in C),
\end{eqnarray*}
the last equality holding by independence between $Z$ and $\{U_{n,i}\}_{n \in \{1,\ldots,M\}, i \in \{1,\ldots,K\}}$. The right hand term goes to 0 when $M \rightarrow \infty$ by $\PAI(Z)$.
This generalizes to bounded functions with compact support, which concludes the proof.
\end{proof}
\subsection{Main result}
Our goal is to show that if $X=\{X^{M}_{n,i}\}$ are asymptotically independent, then $Y=\{Y^{M}_{n,i}\}$ are as well. In other words, if we choose initial conditions that verify a certain property, this property will hold by induction at any finite discrete time.

\begin{theorem}
\label{thm1}
Let $M \in \mathbb{N}$, let $X=\{X^{M}_{n,i}\}_{n \in \{1,\ldots,M\}, i \in \{1,\ldots,K\}}$ be an array of integer valued random variables (the ``state variables''). Suppose that $\PAI(X)$ holds. Then $\PAI(Y)$ holds as well, where $Y$ is defined by \eqref{eq_rmf1}. Moreover, the arrivals to a given node $A^M_{n,i}$  converge in distribution to a Poisson random variable when $M \rightarrow \infty$.
\end{theorem}

Note that the result depends on a choice of initial conditions verifying $\PAI(X)$, a typical example of which is i.i.d. initial conditions stable in law, in the sense that their law does not depend on $M$. The question of whether given an arbitrary initial condition, the dynamics become pairwise asymptotically independent after some (finite or infinite) amount of time, is still open. Note also that this shows that we have convergence in distribution of the exchangeable variables $\{Y^M_n\}$ when $M \rightarrow \infty$.

\section{The proof}\label{sec:thepr}
In the following proof, since $K$ is always finite and all considered random variables are exchangeable, as above, we will sometimes omit the neuron index $i \in \{1,\ldots,K\}$ in order to simplify notation. Tilde superscripts will refer to objects in the infinite replica limit. Hat superscripts will refer to fragmentation processes.

\subsection*{Step one: fragmentation}

\begin{lemma}
\label{lem_hatx_1}
Let $\hat{X}=\{\hat{X}^{M}_{n,i}=X^{M}_{n,i} \one_{\{U_{n,i}>\sigma(X^{M}_{n,i})\}}\}$. Then $\PAI(X)$ implies $\PAI(\hat{X})$.
\end{lemma}

\begin{proof}
We have for $u,v \in [0,1]$,
\begin{equation}
\label{eq_hatx_1}
\Esp[u^{\hat{X}^M_1}v^{\hat{X}^M_2}]=\sum_{k,l \in \mathbb{N}}\P(\hat{X}^M_1=k,\hat{X}^M_2=l)u^k v^l.
\end{equation}
For $k,l >0$, we have
\begin{equation}
\label{eq_hatx_2}
\P(\hat{X}^M_1=k,\hat{X}^M_2=l)=\P(X^M_1=k,X^M_2=l)(1-\sigma(k))(1-\sigma(l)).
\end{equation}
Similarly, 
\begin{eqnarray}
\label{eq_hatx_3}
\P(\hat{X}^M_1=k,\hat{X}^M_2=0) & = & \sum_{l\in \mathbb{N}}\P(X^M_1=k,X^M_2=l)(1-\sigma(k))\sigma(l),\quad \forall k>0\\
\label{eq_hatx_4}
\P(\hat{X}^M_1=0,\hat{X}^M_2=0) & = & \sum_{k,l\in \mathbb{N}}\P(X^M_1=k,X^M_2=l)\sigma(k)\sigma(l).
\end{eqnarray}
Since $\PAI(X)$ holds, for all $k,l \in \mathbb{N}$, $\P(X^M_1=k,X^M_2=l) \rightarrow \P(\tilde{X}=k)\P(\tilde{X}=l)$ when $M \rightarrow \infty.$
Since all considered functions are bounded by 1, we have that for all $k,l \in \mathbb{N}$,
\begin{equation*}
\P(\hat{X}^M_1=k,\hat{X}^M_2=l) \rightarrow \P(\tilde{\hat{X}}=k)\P(\tilde{\hat{X}}=l)
\end{equation*}
when $M \rightarrow \infty,$ where $\tilde{\hat{X}}=\tilde{X}\one_{\{U>\sigma(\tilde{X})\}}.$
This shows that 
\begin{equation}
\label{eq_hatx_5}
\Esp[u^{\hat{X}^M_1}v^{\hat{X}^M_2}] \rightarrow \Esp[u^{\tilde{\hat{X}}}]\Esp[v^{\tilde{\hat{X}}}]
\end{equation}
when $M \rightarrow \infty$, which concludes the proof.

\end{proof}

\subsection*{Step two: asymptotic behavior of the arrivals processes}

We now show that the number of arrivals $A^M_{n,i}$ defined in \eqref{eq_rmf2} is asymptotically Poisson as the number of replicas goes to infinity. This is precisely the Poisson Hypothesis introduced in \cite{kleinrock2007communication}.

\begin{lemma}
\label{lem_A_conv}
Supposing that $\PAI(X)$ holds, when $M \rightarrow \infty$, we have the convergence in distribution
$A^M_{n,i} \rightarrow Poi((K-1)\theta)$
where $\theta=\Esp[\sigma(\tilde{X})]$.
\end{lemma}

\begin{proof}
Let $z\in [0,1]$. Then
\begin{equation*}
\begin{split}
\Esp[z^{A^M_{n,i}}]&=\Esp\left[z^{\sum_{m \neq n}\sum_{j \neq i} \one_{\{U_{m,j}<\sigma(X^{M}_{m,j})\}}\one_{\{R^M_{m,j,i}=n\}}}\right]\\
&=\Esp\left[\prod_{m\neq n}\prod_{j \neq i}\Esp\left[z^{ \one_{\{U_{m,j}<\sigma(X^{M}_{m,j})\}}\one_{\{R^M_{m,j,i}=n\}}}\bigg|X^M_{m,j},U\right]\right] \\
&=\Esp\left[\prod_{m\neq n}\prod_{j \neq i}\left(\left(1-\frac{1}{M-1}\right)+\frac{1}{M-1}z^{\one_{\{U_{m,j}<\sigma(X^{M}_{m,j})\}}}\right)\right] \\
&=\Esp\left[e^{\sum_{m \neq n}\sum_{j \neq i}\log\left(1-\frac{1}{M-1}\left(1-z^{\one_{\{U_{m,j}<\sigma(X^{M}_{m,j})\}}}\right)\right)}\right].
\end{split}
\end{equation*}
We now give an upper and lower bound for this expression.
Since $\log(1-x)\leq -x$ for $x\leq 1$, we have
\begin{equation*}
\Esp\left[z^{A^M_{n,i}}\right]\leq \Esp\left[e^{-\frac{1}{M-1}\sum_{m \neq n}\sum_{j \neq i}\left(1-z^{\one_{\{U_{m,j}<\sigma(X^{M}_{m,j})\}}}\right)}\right].
\end{equation*}
Using the generalized TLLN given in Lemma \ref{lem_gen_tlln}, $\frac{1}{M-1}\sum_{m \neq n}\sum_{j \neq i}(1-z^{\one_{\{U_{m,j}<\sigma(X^{M}_{m,j})\}}}) \rightarrow (K-1)(1-\Phi(z))$ in $L^2$ when $M \rightarrow \infty$ with $\Phi(z)=\Esp[z^{\one_{U<\sigma(\tilde{X})}}],$ where $U$ is any $U_{m,j}$. \\
We have $\Phi(z)=z\int_0^1\P(\sigma(\tilde{X})>t)\diff t+(1-\int_0^1\P(\sigma(\tilde{X})>t)\diff t)=(z-1)\theta+1.$

Therefore, since $L^2$ convergence implies convergence in distribution and thus convergence of the Laplace transforms, 
\begin{equation*}
\Esp\left[e^{-\frac{1}{M-1}\sum_{m \neq n}\sum_{j \neq i}\left(1-z^{\one_{\{U_{m,j}<\sigma(X^{M}_{m,j})\}}}\right)}\right] \rightarrow e^{-\theta(1-z)(K-1)}
\end{equation*}
when $M \rightarrow \infty$. Thus,
\begin{equation}
\label{eq_tllnA_1}
\limsup_{M\rightarrow \infty}\Esp[z^{A^M_{n,i}}] \leq e^{-\theta(1-z)(K-1)}.
\end{equation}

Similarly, since $\log(1-x)\geq -x-\frac{x^2}{2}$ for $x\leq 1$, we have
\begin{equation*}
\begin{split}
\Esp\left[z^{A^M_{n,i}}\right]\geq & \Esp\bigg[e^{-\frac{1}{M-1}\sum_{m \neq n}\sum_{j \neq i}\left(1-z^{\one_{\{U_{m,j}<\sigma(X^{M}_{m,j})\}}}\right)}\\
& e^{-\frac{1}{2(M-1)^2}\sum_{m \neq n}\sum_{j \neq i}\left(1-z^{\one_{\{U_{m,j}<\sigma(X^{M}_{m,j})\}}}\right)^2}\bigg].
\end{split}
\end{equation*}
Using once again Lemma \ref{lem_gen_tlln}, as the second term goes to 0 when $M \rightarrow \infty$, by the same reasoning as previously, we get 
\begin{equation}
\label{eq_tllnA_2}
\liminf_{M\rightarrow \infty}\Esp[z^{A^M_{n,i}}] \geq e^{-\theta(1-z)(K-1)}.
\end{equation}
Combining \eqref{eq_tllnA_1} and \eqref{eq_tllnA_2}, the result follows.
\end{proof}

Now, we show that the arrivals to different replicas become pairwise asymptotically independent:

\begin{lemma}
\label{lem_PAI_A}
For all $(n,i) \neq (m,j)$, $A^M_{n,i}$ and $A^M_{m,j}$ are pairwise asymptotically independent.
\end{lemma}
\begin{proof}
We first show the result in the case $n \neq m$ and $i \neq j$. 
Let $u,v \in [0,1].$
Then
\begin{equation*}
\begin{split}
\Esp\left[u^{A^M_{n,i}}v^{A^M_{m,j}}\right]=&\Esp\bigg[u^{\sum_{n' \neq n, i' \neq i}\one_{\{U_{n',i'}<\sigma(X^M_{n',i'})\}}\one_{\{R^M_{n',i',i}=n\}}} \\
&v^{\sum_{m' \neq m, j' \neq j}\one_{\{U_{m',j'}<\sigma(X^M_{m',j'})\}}\one_{\{R^M_{m',j',j}=m\}}}\bigg] \\
=&\Esp\bigg[\prod_{n' \neq n, i' \neq i}u^{\one_{\{U_{n',i'}<\sigma(X^M_{n',i'})\}}\one_{\{R^M_{n',i',i}=n\}}}\\
&\prod_{m' \neq m, j' \neq j}v^{\one_{\{U_{m',j'}<\sigma(X^M_{m',j'})\}}\one_{\{R^M_{m',j',j}=m\}}}\bigg] \\ 
=&\Esp\Bigg[\Esp\bigg[\prod_{n' \neq n, i' \neq i}u^{\one_{\{U_{n',i'}<\sigma(X^M_{n',i'})\}}\one_{\{R^M_{n',i',i}=n\}}}\\
&\prod_{m' \neq m, j' \neq j}v^{\one_{\{U_{m',j'}<\sigma(X^M_{m',j'})\}}\one_{\{R^M_{m',j',j}=m\}}}\bigg| X^M, U\bigg]\Bigg] \\ 
=&\Esp\bigg[\prod_{n' \neq n, i' \neq i}\left[\left(1-\frac{1}{M-1}\right)+\frac{1}{M-1}u^{\one_{\{U_{n',i'}<\sigma(X^M_{n',i'})\}}}\right] \\
& \prod_{m' \neq m, j' \neq j}\left[\left(1-\frac{1}{M-1}\right)+\frac{1}{M-1}v^{\one_{\{U_{m',j'}<\sigma(X^M_{m',j'})\}}}\right]\bigg] \\
&=\Esp\bigg[e^{\sum_{n' \neq n, i' \neq i}\log\left(1-\frac{1}{M-1}\left(1-u^{\one_{\{U_{n',i'}<\sigma(X^M_{n',i'})\}}}\right)\right)}\\
&e^{\sum_{m' \neq m, j' \neq j}\log\left(1-\frac{1}{M-1}\left(1-v^{\one_{\{U_{m',j'}<\sigma(X^M_{m',j'})\}}}\right)\right)}\bigg].
\end{split}
\end{equation*}
The fourth equality above comes from the independence between the routing variables $R^M$.

Just as in the proof of Lemma \ref{lem_A_conv}, we can give upper and lower bounds of the last right-hand side expression:
\begin{equation*}
\Esp\left[u^{A^M_{n,i}}v^{A^M_{m,j}}\right]\leq \Esp\left[e^{-\frac{1}{M-1}\sum_{n' \neq n, i' \neq i}\left(2-u^{\one_{\{U_{n',i'}<\sigma(X^M_{n',i'})\}}}-v^{\one_{\{U_{n',i'}<\sigma(X^M_{n',i'})\}}}\right)}\right]
\end{equation*}
and
\begin{equation*}
\begin{split}
\Esp\left[u^{A^M_{n,i}}v^{A^M_{m,j}}\right] &\geq \Esp\bigg[e^{-\frac{1}{M-1}\sum_{n' \neq n, i' \neq i}\left(2-u^{\one_{\{U_{n',i'}<\sigma(X^M_{n',i'})\}}}-v^{\one_{\{U_{n',i'}<\sigma(X^M_{n',i'})\}}}\right)} \\
& \hspace{1cm} \cdot e^{-\frac{1}{2(M-1)^2}\sum_{m \neq n}\sum_{j \neq i}\left(2-u^{\one_{\{U_{n',i'}<\sigma(X^M_{n',i'})\}}}-v^{\one_{\{U_{n',i'}<\sigma(X^M_{n',i'})\}}}\right)^2}\bigg]. \\
\end{split}
\end{equation*}
The last right-hand side expression goes to $e^{(1-u+1-v)(K-1)\theta}$ when $M \rightarrow \infty$ in both cases, as previously.
The result follows from these two bounds as in the proof of Lemma \ref{lem_A_conv}.

The case where $n=m$, i.e., when we consider the arrivals to two different neurons in the same replica, is done in the same way since the routing variables are independent from the neurons chosen.
The case where $i=j$, i.e. when we consider the arrivals to the same neuron in two different replicas, is treated in the same way, with the extra step of isolating the terms that are not independent from each other.
\end{proof}

\subsection*{Step three: propagation of pairwise asymptotic independence}
Our goal is now to combine the previous results to show that $\PAI(Y)$ holds, assuming $\PAI(X)$. 
We have that for all $i \in \{1,\ldots,K\}$ and all $n \in \{1,\ldots,M\}$, $Y^M_{n,i}=\hat{X}^M_{n,i}+A^M_{n,i}$. We call $\tilde{A}$ the limit in distribution of $A^M_{n,i}$ (it is Poisson distributed by the previous lemma).
It is clear that by exchangeability between replicas, we only require the following lemma:
\begin{lemma}
\label{lem_pai_y}
For $i,j \in \{1,\ldots,K\}$,
\begin{equation}
\label{eq_paiy_1}
\Esp[u^{Y^M_{1,i}},v^{Y^M_{2,j}}] \rightarrow \Esp[u^{\tilde{Y}}]\Esp[v^{\tilde{Y}}]
\end{equation}
when $M \rightarrow \infty$, where $\tilde{Y}=\tilde{\hat{X}}+\tilde{A}.$
\end{lemma}
\begin{proof}
Let $u,v \in [0,1]$. Then, given $i,j \in [0,K]$, with $i \neq j$ for simplicity,
\begin{equation*}
\begin{split}
\Esp\left[u^{Y^M_{1,i}}v^{Y^M_{2,j}}\right]=&\Esp\bigg[u^{\hat{X}^M_{1,i}}v^{\hat{X}^M_{2,j}}u^{\sum_{n' \neq 1, i' \neq i}\one_{\{U_{n',i'}<\sigma(X^M_{n',i'})\}}\one_{\{R^M_{n',i',i}=1\}}}\\
&v^{\sum_{m' \neq 2, j' \neq j}\one_{\{U_{m',j'}<\sigma(X^M_{m',j'})\}}\one_{\{R^M_{m',j',j}=2\}}}\bigg] \\
=&\Esp\bigg[u^{\hat{X}^M_{1,i}}v^{\hat{X}^M_{2,j}}\prod_{n' \neq 1, i' \neq i}u^{\one_{\{U_{n',i'}<\sigma(X^M_{n',i'})\}}\one_{\{R^M_{n',i',i}=1\}}}\\
&\prod_{m' \neq 2, j' \neq j}v^{\one_{\{U_{m',j'}<\sigma(X^M_{m',j'})\}}\one_{\{R^M_{m',j',j}=2\}}}\bigg] \\ 
=&\Esp\Bigg[\Esp\bigg[u^{\hat{X}^M_{1,i}}v^{\hat{X}^M_{2,j}}\prod_{n' \neq 1, i' \neq i}u^{\one_{\{U_{n',i'}<\sigma(X^M_{n',i'})\}}\one_{\{R^M_{n',i',i}=1\}}}\\
&\prod_{m' \neq 2, j' \neq j}v^{\one_{\{U_{m',j'}<\sigma(X^M_{m',j'})\}}\one_{\{R^M_{m',j',j}=2\}}}\bigg|X^M,U\bigg]\Bigg]\\
=&\Esp\bigg[u^{\hat{X}^M_{1,i}}v^{\hat{X}^M_{2,j}}\prod_{n' \neq 1, i' \neq i}\left(\frac{1}{M-1} u^{\one_{\{U_{n',i'}<\sigma(X^M_{n',i'})\}}}+\left(1-\frac{1}{M-1}\right)\right)
\\
&\prod_{m' \neq 2, j' \neq j}\left(\frac{1}{M-1}v^{\one_{\{U_{m',j'}<\sigma(X^M_{m',j'})\}}}+\left(1-\frac{1}{M-1}\right)\right)\bigg] \\
=&\Esp\left[\phi_1^M(u,v)\phi_2^M(u,v)\right],
\end{split}
\end{equation*}
where 
\begin{equation*}
\begin{split}
\phi_1^M(u,v)=&u^{\hat{X}^M_{1,i}}\left(1-\frac{1}{M-1}+\frac{1}{M-1}v^{\one_{\{U_{1,i}<\sigma(X^M_{1,i})\}}}\right)\\
&v^{\hat{X}^M_{2,j}}\left(1-\frac{1}{M-1}+\frac{1}{M-1}u^{\one_{\{U_{2,j}<\sigma(X^M_{2,j})\}}}\right)
\end{split}
\end{equation*}
and
\begin{equation*}
\begin{split}
\phi_2^M(u,v)=&e^{\sum_{n' \neq 1; i' \neq i;(n',i')\neq (2,j)}\log\left(1-\frac{1}{M-1}\left(1-u^{\one_{\{U_{n',i'}<\sigma(X^M_{n',i'})\}}}\right)\right)} \\
&e^{\sum_{m' \neq 2; j' \neq j;(m',j')\neq (1,i)}\log\left(1-\frac{1}{M-1}\left(1-v^{\one_{\{U_{m',j'}<\sigma(X^M_{m',j'})\}}}\right)\right)}.
\end{split}
\end{equation*}

When $M \rightarrow \infty$, by Lemmas \ref{lem_hatx_1} and \ref{lem_PAI_A}, $\phi_1^M(u,v)$ and $\phi_2^M(u,v)$ are pairwise asymptotically independent. Since in $\phi_2^M(u,v),$ the contribution of the missing terms in the sum is negligible, when $M \rightarrow \infty,$ we have
\begin{equation}
\label{eq_paiy_2}
\Esp\left[\phi_1^M(u,v)\phi_2^M(u,v)\right] \rightarrow \Esp\left[u^{\tilde{\hat{X}}}\right]\Esp\left[v^{\tilde{\hat{X}}}\right]\Esp\left[u^{\tilde{\hat{A}}}\right]\Esp\left[v^{\tilde{\hat{A}}}\right].
\end{equation}
This shows that \eqref{eq_paiy_1} holds.
\end{proof} 
Thus, $\PAI(X)$ implies $\PAI(Y)$, which concludes the proof of the theorem. Note that Lemma \ref{lem_pai_y} also shows that $\tilde{\hat{X}}$ and $\tilde{A}$ are independent.

\section{The symmetric fragmentation-interaction-aggregation process}\label{sec:sfiap}
Our goal is to show that propagation of chaos and the Poisson hypothesis hold in the more general setting of symmetric FIAPs under mild hypotheses on the dynamics of the system. The symmetrical evolution equations read


\begin{equation}
\label{eq_sym1}
Y_{i}=g_1(X_{i})\one_{\{U_{i}<\sigma(X_{i})\}}+g_2(X_{i})\one_{\{U_{i}>\sigma(X_{i})\}}+A_{i}
\end{equation}
where
\begin{equation}
\label{eq_sym2}
A_{i}=\sum_{j \neq i} h(X_j)\one_{\{U_{j}<\sigma(X_{j})\}}
\end{equation}
and $g_1,g_2,h : \mathbb{N} \rightarrow \mathbb{N}$ are functions such that $h$ is bounded.


We now introduce the corresponding replica dynamics.
Let $\{X^{M}_{n,i}\}$ be the integer-valued state variables at step 0, where $n \in \{1,\ldots,M\}$ and $i \in \{1,\ldots,K\}.$ Let $\{Y^{M}_{n,i}\}$ be the integer-valued state variables at time one.
Let $\{U_{n,i}\}$ be i.i.d. random variables independent from $\{X^{M}_{n,i}\}$ uniformly distributed in $[0,1]$. We introduce again
the i.i.d. routing variables $R^M_{m,j,i}$, independent from $\{U_{n,i}\}$ and $\{X^{M}_{n,i}\}$ and uniformly distributed in $\{1,\ldots,M\}\setminus\{m\}$.
The $M$-replica equations read:

\begin{equation}
\label{eq_rmf_sym1}
Y^{M}_{n,i}=g_1(X^{M}_{n,i})\one_{\{U_{n,i}<\sigma(X^{M}_{n,i})\}}+g_2(X^{M}_{n,i})\one_{\{U_{n,i}>\sigma(X^{M}_{n,i})\}}+A^M_{n,i},
\end{equation}
where
\begin{equation}
\label{eq_rmf_sym2}
A^M_{n,i}=\sum_{m \neq n}\sum_{j \neq i} h(X^M_{m,j})\one_{\{U_{m,j}<\sigma(X^{M}_{m,j})\}}\one_{\{R^M_{m,j,i}=n\}}
\end{equation}
is the number of arrivals in node $i$ of replica $n$.


We also recall the definition of a compound Poisson distribution:
\begin{definition}
\label{def_compP}
The random variable $X$ is said to follow a compound Poisson distribution if there exist a Poisson($\lambda$) random variable $N$ and i.i.d. random variables $(X_i)_{i \in \mathbb{N}^\star}$ independent from $N$ such that
$X=\sum_{i=1}^N X_i.$
The generating function of $X$, denoted $\phi_X$, is given by 
\begin{equation}
\label{eq_pgf_compoundp}
\phi_X(t)=e^{\lambda(\phi(t)-1)},
\end{equation}
where $\phi(t)$ is the generating function of $X_1$.
\end{definition}

We have the following theorem:
\begin{theorem}
\label{thm2}
For all symmetric RMF FIAP dynamics, $\PAI(X)$ implies $\PAI(Y)$. Moreover, the arrivals to a given node are asymptotically compound Poisson distributed.
\end{theorem}

We will require the following lemmas. 
The following result replaces Lemma \ref{lem_hatx_1}:
\begin{lemma}
\label{lem_hat_2}
Let $\hat{X^1}=\{\hat{X}^{1,M}_{n,i}=g_1(X^{M}_{n,i})\one_{\{U_{n,i}<\sigma(X^{M}_{n,i})\}}\}.$ \\
Let $\hat{X^2}=\{\hat{X}^{2,M}_{n,i}=g_2(X^{M}_{n,i})\one_{\{U_{n,i}>\sigma(X^{M}_{n,i})\}}\}.$
Then $\PAI(X)$ implies $\PAI(\hat{X^1})$, $\PAI(\hat{X^2})$ and $\PAI(\hat{X})$, where $\hat{X}=\hat{X^1}+\hat{X^2}$.
\end{lemma}
\begin{proof}
We proceed exactly as in Lemma \ref{lem_hatx_1}. We write here only the proof for $\hat{X^2}$, the others being identical except for the numerical expressions involved.
We have for $u,v \in [0,1]$,
\begin{equation}
\label{eq_sym_hatX_1}
\Esp[u^{\hat{X}^{2,M}_1}v^{\hat{X}^{2,M}_2}]=\sum_{k,l \in \mathbb{N}}\P(\hat{X}^{2,M}_1=k,\hat{X}^{2,M}_2=l)u^{k}v^{l}.
\end{equation}
For $k,l >0$, we have 
\begin{equation}
\label{eq_sym_hatX_2}
\begin{split}
\P(\hat{X}^{2,M}_1=k,\hat{X}^{2,M}_2=l)=&\sum_{p,q \in \mathbb{N}}\P(g_2(X^M_1)=k,g_2(X^M_2)=l, X^M_1=p, X^M_2=q)\\
&(1-\sigma(p))(1-\sigma(q)).
\end{split}
\end{equation}
Since $\PAI(X)$ holds, 
\begin{equation*}
\P(g_2(p)=k,g_2(q)=l, X^M_1=p, X^M_2=q) \rightarrow \P(\tilde{X}=p, g_2(p)=k)\P(\tilde{X}=q, g_2(q)=l)
\end{equation*} 
when $M \rightarrow \infty.$
Hence, $\Esp[u^{\hat{X}^{2,M}_1}v^{\hat{X}^{2,M}_2}] \rightarrow \Esp\left[u^{g_2(\tilde{X})\one_{\{U<\sigma(\tilde{X})\}}}\right]\Esp\left[v^{g_2(\tilde{X})\one_{\{U<\sigma(\tilde{X})\}}}\right]$ when $M \rightarrow \infty$.
The cases where $k$ and/or $l$ are equal to 0 are handled in the same way. This proves the result.
\end{proof}
 
We now prove the following result, which replaces Lemma \ref{lem_A_conv}:
\begin{lemma}
\label{lem_conv_A_2}
Supposing that $\PAI(X)$ holds, when $M \rightarrow \infty,$ we have the convergence in distribution
$A^M_{n,i} \rightarrow \tilde{A}$,
where $\tilde{A}$ follows a compound Poisson distribution.
\end{lemma}
\begin{proof}
We still have, just like in the proof of Lemma \ref{lem_A_conv}, that for $z\in [0,1], i \in \{1,\ldots,K\}, n \in \{1,\ldots,M\}$,
\begin{equation}
\label{eq_sym_A_1}
\Esp\left[z^{A^M_{n,i}}\right]=\Esp\left[e^{\sum_{m \neq n}\sum_{j \neq i}\log\left(1-\frac{1}{M-1}\left(1-z^{h(X^M_{m,j})\one_{\{U_{m,j}<\sigma(X^{M}_{m,j})\}}}\right)\right)}\right].
\end{equation}
Using the same arguments as before, we have when $M \rightarrow \infty$
\begin{equation}
\label{eq_sym_A_2}
\Esp\left[z^{A^M_{n,i}}\right] \rightarrow e^{(K-1)(\Phi(z)-1)},
\end{equation}
where $\Phi(z)=\Esp\left[z^{h(\tilde{X})\one_{\{U<\sigma(\tilde{X})\}}}\right]$,
which is precisely of the form \eqref{eq_pgf_compoundp}, that is, a generating function of a random variable with a compound Poisson distribution.
\end{proof}
We now combine these results to prove Theorem \ref{thm2}.
\begin{proof}
We follow the outline of the previous section. Lemmas \ref{lem_L2}, \ref{lem_pai_tlln} and \ref{lem_gen_tlln} still apply as previously. Lemma \ref{lem_hat_2} replaces Lemma \ref{lem_hatx_1}. Lemma \ref{lem_conv_A_2} replaces Lemma \ref{lem_A_conv}. Lemmas \ref{lem_PAI_A} and \ref{lem_pai_y} still hold, with only differences in the limiting expressions.
\end{proof}

\section{The general fragmentation-interaction-aggregation process}\label{sec:gfiap}
The previously introduced exchangeable dynamics allow for simpler computations at the expense of realistic modeling. For example, neuron populations are not homogeneous and are not fully connected. In order to account for such a geometry, we now generalize the previous result to the case where the functions governing the information received by a node when another node activates depend on the nodes involved. Specifically, recall the class $\Cset$ of discrete FIAP defined in Section 1.

For any process in $\Cset$, we can define a replica mean field model as in the previous sections: we consider a collection of $M$ identically distributed replicas of a set of $K$ nodes, which could be neurons, particles, queues or other objects, depending on context.
As previously, let $\{X^{M}_{n,i}\}$ be the integer-valued state variables at step 0, where $n \in \{1,\ldots,M\}$ and $i \in \{1,\ldots,K\}.$ Let $\{Y^{M}_{n,i}\}$ be the integer-valued state variables at time one.
Let $\{U_{n,i}\}$ be uniformly distributed on $[0,1]$ i.i.d. random variables independent from $\{X^{M}_{n,i}\}$. Let $\{R^M_{m,j,i}\}$ be i.i.d. routing random variables independent from $\{X^{M}_{n,i}\}$ and $\{U_{n,i}\}$, uniformly distributed on  $\{1,\ldots,M\}\setminus\{m\}$ for all $i,j \in \{1,\ldots,K\}$ and $m \in \{1,\ldots,M\}.$
Recall that the $M$-RMF equations read
\begin{equation}
\label{eq_gen_rmf_1_bis}
Y^{M}_{n,i}=g_{1,i}(X^{M}_{n,i})\one_{\{U_{n,i}<\sigma_i(X^{M}_{n,i})\}}+g_{2,i}(X^{M}_{n,i})\one_{\{U_{n,i}>\sigma_i(X^{M}_{n,i})\}}+A^M_{n,i},
\end{equation}
where
\begin{equation}
\label{eq_gen_rmf_2_bis}
A^M_{n,i}=\sum_{m \neq n}\sum_{j \neq i} h_{i,j}(X^M_{m,j})\one_{\{U_{m,j}<\sigma_i(X^{M}_{m,j})\}} \one_{\{R^M_{m,j,i}=n\}}
\end{equation}
is the number of arrivals in node $i$ of replica $n$. We now show that the result from the previous section carries over to this more general setting with only minor modifications. 

First, we must slightly modify the definition of pairwise asymptotic independence in order to take into account the dependence on the node of the limiting distribution. As a simplification, we keep the same notations for this modified definition.
\begin{definition}
\label{def_pai_gen}
Given $M \in \mathbb{N}$, given an array of integer-valued random variables $Z=\{Z^{M}_{n,i}\}_{1 \leq n \leq M, 1 \leq i \leq K}$ such that for all fixed $M$, the random variables $Z^M_{n,i}$ are exchangeable in $n$, we say that the variables $Z^M_{n,i}$ are pairwise asymptotically independent, which we will denote $\PAI(Z)$, if there exist integer-valued random variables $(\tilde{Z}_i)_{i \in \{1,\ldots,K\}}$ such that $\forall (n,i) \neq (m,j),\forall u,v \in [0,1]$,
\begin{equation}
\label{eq_gen_pai_1}
\lim_{M \rightarrow \infty}\Esp[u^{Z^M_{n,i}}v^{Z^M_{m,j}}] =\Esp[u^{\tilde{Z}_i}]\Esp[v^{\tilde{Z}_j}].
\end{equation}
\end{definition}
For clarity of exposition, we also recall here the definition of the triangular law of large numbers, even though it is left unchanged:
\begin{definition}
\label{def_tlln_gen}
Given $M \in \mathbb{N}$, given an array of integer-valued random variables $Z=\{Z^{M}_{n}\}_{n \in \{1,\ldots,M\}}$ such that for all fixed $M$, the random variables $Z^M_{n}$ are exchangeable in $n$, we say that $Z$ verifies the triangular law of large numbers $\TLLN(Z)$ if there exist an integer-valued random variable $\tilde{Z}$ such that for all functions $f:\mathbb{N} \rightarrow \mathbb{R}$ with compact support, we have the following limit in $L^2$:
\begin{equation}
\label{eq_gen_tlln_1}
\lim_{M\rightarrow\infty}\frac{1}{M}\sum_{n=1}^M f(Z_{n}^M)=\Esp[f(\tilde{Z})].
\end{equation} 
\end{definition}

Then, we obtain the same result:

\begin{theorem}
\label{thm3}
Using previously defined notations, $\PAI(X)$ implies $\PAI(Y)$. Moreover, the arrivals to a given node are asymptotically compound Poisson distributed and are independent of the states of the nodes.
\end{theorem}

We once again require the following lemmas for the proof.

We replace Lemma \ref{lem_gen_tlln} with the following similar result, taking into account the fact that the limiting distribution now depends on the node:

\begin{lemma}
\label{lem_gen_tlln_gen}
Let $M \in \mathbb{N}$, let $Z=\{Z^{M}_{n,i}\}_{n \in \{1,\ldots,M\}, i \in \{1,\ldots,K\}}$ be an array of integer valued random variables verifying $\PAI(Z)$. Then for all bounded functions $f : \mathbb{N}\times [0,1] \rightarrow \mathbb{R}$ with compact support, for all i.i.d. sequences of random variables $U=\{U_{n,i}\}_{n \in \{1,\ldots,M\}, i \in \{1,\ldots,K\}}$ independent from $Z$, there exists $U$ independent from $(\tilde{Z}_i)_{i \in \{1,\ldots,K\}}$ and $Z$ such that for all $i \in \{1,\ldots,K\}$, we have the following limit in $L^2$:
\begin{equation}
\label{eq_modif_gen_tlln_1}
\lim_{M\rightarrow\infty}\frac{1}{M}\sum_{n=1}^M f(Z_{n,i}^M, U_{n,i})=\Esp[f(\tilde{Z}_i,U)].
\end{equation}
\end{lemma}
The proof is exactly the same as for Lemma \ref{lem_pai_tlln}.

We must replace Lemma \ref{lem_A_conv} with the following result:
\begin{lemma}
\label{lem_A_conv_gen}
Supposing that $\PAI(X)$ holds, when $M \rightarrow \infty$:
$A^M_{n,i} \rightarrow \tilde{A}_i$ in distribution,
where $\tilde{A}_i$ follows a compound Poisson distribution.
\end{lemma}
\begin{proof}
We have for $z\in [0,1], i \in \{1,\ldots,K\}, n \in \{1,\ldots,M\}$, that
\begin{equation}
\label{eq_gen_A_1}
\Esp\left[z^{A^M_{n,i}}\right]=\Esp\left[e^{\sum_{m \neq n}\sum_{j \neq i}\log\left(1-\frac{1}{M-1}\left(1-z^{h_{i,j}(X^M_{m,j})\one_{\{U_{m,j}<\sigma_i(X^{M}_{m,j})\}}}\right)\right)}\right].
\end{equation}
Using the same arguments as before, we have when $M \rightarrow \infty$
\begin{equation}
\label{eq_gen_A_2}
\Esp\left[z^{A^M_{n,i}}\right] \rightarrow e^{\Phi_i(z)},
\end{equation}
where $\Phi_i(z)=-\sum_{j \neq i}\Esp\left[1-z^{h_{i,j}(\tilde{X_i})\one_{\{U<\sigma_i(\tilde{X}_i)\}}}\right]$.
Therefore,
\begin{equation}
\label{eq_gen_A_3}
\Esp\left[z^{A^M_{n,i}}\right] \rightarrow e^{-\sum_{j \neq i}\left(1-\Esp\left[z^{h_{i,j}(\tilde{X}_i)\one_{\{U<\sigma_i(\tilde{X}_i)\}}}\right]\right)}.
\end{equation}
The expression is of the form \eqref{eq_pgf_compoundp}, which proves Lemma \ref{lem_A_conv_gen}.
\end{proof}

We now prove Theorem \ref{thm3}.
\begin{proof}
We use the same reasoning as previously.

Lemma \ref{lem_pai_tlln} still holds (the replicas are still exchangeable, only the nodes are not). Lemma \ref{lem_gen_tlln_gen} replaces Lemma \ref{lem_gen_tlln}. Since the functions $g_{1,i}$ and $g_{2,i}$ only depend on the node and not on the replica index, an equivalent result to Lemma \ref{lem_hat_2} still holds. Lemma \ref{lem_A_conv} is replaced by Lemma \ref{lem_A_conv_gen}.
For asymptotic independence, we have, using the same arguments as in the proof of Lemma \ref{lem_PAI_A}, that for $u,v \in [0,1]$, for $n \neq m$ and $i \neq j$,
\begin{equation}
\label{eq_A_AI_1}
\Esp\left[u^{A^M_{n,i}}v^{A^M_{m,j}}\right] \rightarrow e^{-\sum_{i' \neq i}\left(1-\Esp\left[u^{h_{i,i'}(\tilde{X}_i)\one_{\{U<\sigma_i(\tilde{X}_i)\}}}\right]\right)-\sum_{j' \neq j}\left(1-\Esp\left[v^{h_{j,j'}(\tilde{X}_j)\one_{\{U<\sigma_i(\tilde{X}_j)\}}}\right]\right)},
\end{equation}
when $M \rightarrow \infty$. The other cases ($n=m$ and $i=j$) are also valid.
Lemma \ref{lem_pai_y} also still holds, with only minor differences in the limit expressions.
\end{proof}
Note that once again, this proves that the limit processes $\tilde{\hat{X}}_i$ and $\tilde{A}_j$ are independent for all $i,j \in \{1,\ldots,K\}.$

As an application, let us apply this result to the model from Section 2 with the addition of nonexchangeable interactions. Namely, we consider $h_{i,j}(X^M_{m,j})=\mu_{i,j}$ with $\mu_{i,j} \in \mathbb{N}$ (potentially zero). In this case, Theorem \ref{thm3} proves the propagation of chaos in this system, and the limit distributions of arrivals at the different nodes are characterized by, for $i \in \{1,\ldots,K\}$ and $z \in [0,1]$,
\begin{equation}
\label{eq_example_1}
\Esp\left[z^{\tilde{A}_i}\right] = e^{\theta_i\sum_{j \neq i}\left(z^{\mu_{i,j}}-1\right)} = \prod_{j \neq i} e^{\theta_i\left(z^{\mu_{i,j}}-1\right)},
\end{equation}
where $\theta_i=\Esp\left[\sigma_i(\tilde{X}_i)\right]$.
Note that as expected, when all $\mu_{i,j}$ are equal to one, we obtain the result from Section 2.

\section{Extensions}\label{sec:ext}
There are several ways of extending the FIAP framework while preserving the
basic properties proved in the present paper (propagation of chaos and Poisson hypothesis).
We decided not to include them in the general framework in order to keep notation and exposition light.
A few natural extensions of this type are nevertheless discussed below.

\paragraph{Random Interactions} The functions $h_{i,j}(k)$ can be
replaced by randomized functions of the type $h_{i,j}(k,V_{i,j})$ where
the random variables $\{V_{i,j}\}_{1\le i,j\le K}$ are uniform in $[0,1]$
and i.i.d.. This allows one to represent, e.g., the queuing theory scenario where a
customer leaving a queue is randomly routed to an other queue of the network according to some
stochastic routing matrix $\{p_{j,i}\}_{1\le i,j \le K}$, namely a customer
leaving queue $i$ is routed to queue $j$ with probability $p_{j,i}$. If
the random variables $\{V_{i,j}\}$ are independent of $\{X_i\}_i$, then the main results still hold.

\paragraph{Time in-homogeneous dynamics} The general setting of the paper implicitly suggests to use the same (activation, fragmentation, and interaction) functions
at all time steps for a given node. There is no difficulty extending the results to the time in-homogeneous case where these functions depend on the time step. In the neural network case, this for instance happens in certain learning dynamics where the synaptic weights evolve over time.

\paragraph{Exogenous input and output} 
To the {\em endogenous} arrivals $A_i$ to node $i$
given in Equation (\ref{eq_gendef_1}), we add arbitrary {\em exogenous} arrivals $B_i\in \mathbb N$. In the special care where the random 
variables $\{B_i\}$ are independent, Poisson, and independent of the state variables $\{X_i\}_i$
then the same results still hold. Note that one can also define an {\em exogenous output} for
node $i$ through the relation
$$D_i= h_{o,i} (X_i) \one_{\{U_i<\sigma_i(X_i)\}}\, ,$$
where $h_{o,i}$ is a given output function $\mathbb N\to \mathbb N$.

Exogenous input point processes are needed for modeling reasons
in a variety of contexts (e.g., to represent requests from end-users
in a computer network, or input signals from sensors in a neural network).
Exogenous output processes are useful in, e.g., a two-layer network where the first layer feeds the second one, but not conversely.
Exogenous input and output point processes are instrumental in the partition scheme of the vector state example discussed below.
In that example, the exogenous input variables are neither necessarily Poisson nor independent.

\paragraph{Vector State - Example} This extension is first described through a
simple neural network example. We partition the set of neurons of a discrete
Galves-L\"{o}cherbach network in pairs (this assumes that $K$ is even).
Each pair of the partition is a node of the network. If $(i,j)$ is one of these nodes,
it has a two-dimensional vector state $(X_i,X_j)$ (rather than a one dimensional state in the
initial model). We let this pair (as well as each other pair in the partition) evolve as
a two-node  Galves-L\"{o}cherbach network with some {\em vector} exogenous input and output.
More precisely, conditionally on $(X_i,X_j)=(k,l)$, neurons $i$ and $j$
spike independently with probability $\sigma_i(k)$ and $\sigma_j(l)$ respectively.
If none of them spikes, the state $(Z_i,Z_j)$ after its endogenous evolution is still $(X_i,X_j)$.
If only $i$ (resp. $j$) spikes and the other neuron of the pair does not spike, then $(Z_i,Z_j)$
is equal $(0,X_j+r_{j,i})$ (resp. $(X_i+r_{i,j},0)$). If both spike, then $(Z_i,Z_j)=(r_{i,j},r_{j,i})$,
with $r_{k,l}$ equal to 1 of there is a directed edge from $l$ to $k$ and 0 otherwise.
Therefore, if $(B_i,B_j)$ denotes the vector exogenous input,
the state of this pair at time one is $(Z_i+B_i,Z_j+B_j)$, by combining the endogenous evolution
and the exogenous arrivals. 

Define now the {\em exogenous output of type $k\notin \{i,j\}$
of node $(i,j)$} by
\begin{equation}
D_k(i,j)= \one_{\{U_i<\sigma_i(X_i)\}} r_{k,i} + \one_{\{U_j<\sigma_j(X_j)\}} r_{k,j} ,
\end{equation}
The extension of interest here is that where we take the following exogenous input to node $(i,j)$:
\begin{equation}
B_i=\sum_{k\notin \{i,j\}} D_i(k,l(k))\one_{\{k<l(k)\}}, \qquad
B_j=\sum_{k\notin \{i,j\}} D_j(k,l(k))\one_{\{k<l(k)\}},
\end{equation}
with $l(k)$ the neuron paired with $k$. 

Note that with this definition, when neuron $i$ spikes, the effect
on pair $(k,l)$ with $l=l(k)$ is as follows: no effect if $r_{k,i}=r_{l,i}=0$; one arrival in $k$ and none in $l$
(resp. one in $l$ and none in $k$) if $r_{k,i}=1$ and $r_{l,i}=0$ (resp.
$r_{l,i}=1$ and $r_{k,i}=0$); a simultaneous arrival in both $l$ and $k$ otherwise.
This defines a network which does not belong to the FIAP class, were it only because the state is now a vector.
The $M$-RMF model features $M$ replicas of this network with $K/2$ (vector state) nodes each.
In this $M$-RMF model, the exogenous output of node/pair $(i,j)$ in replica $m$ is randomly sent to a replica chosen at random. More precisely, for all exogenous output type $k$ paired with $l$,
$$D_k^m(i,j)= \one_{\{U_i^m<\sigma_i(X^m_i)\}} r_{k,i} + \one_{\{U^m_j<\sigma_j(X^m_j)\}} r_{k,i},$$
$${\mathrm {(resp.}}\ D_l^m(i,j)= \one_{\{U_i^m<\sigma_i(X^m_i)\}} r_{l,i} +
\one_{\{U^m_j<\sigma_j(X^m_j)\}} r_{l,i}) $$
units are sent to $k$ (resp. $l$) of another replica selected uniformly at random
where they are aggregated to the coordinates of the state variable of this pair.
It can be shown that when $M$ tends to infinity, (1)
the random state vectors
$(X^m_{i},X^m_{j})$ and $(X^m_{i'},X^m_{j'})$, where $(i,j)$ and $(i',j')$ are
two different pairs, are asymptotically independent (although the two coordinates
of each vector are in general dependent);
(2) the
exogenous arrivals to any coordinate of a pair in a typical
replica tends to an independent compound Poisson variable.

\paragraph{Vector State - General Case}
Consider a FIAP $F$ with $K$ nodes. Let $S_1,S_2,\ldots,S_l$ be a partition
of $[1,\ldots,K]$. Let $K_p$, $1\le p\le l$ denote the cardinality of set $S_p$,
and let $F_p$ be the restriction of $F$ to the coordinates of $S_p$.
Let $\mathfrak{F}_p$ be the FIAP combining the endogenous dynamics of $F_p$
and exogenous input $(\mathfrak{B}_{p,i},\ i\in S_p)$. Let
$\mathfrak{X}_{p,i}$ denote the state variables in $\mathfrak{F}_p$.
For all $k\notin S_p$, define the {\em exogenous output} of type $k$ of $\mathfrak{F}_p$ as
\begin{equation}
\label{eq:gendep}
\mathfrak{D}_p(k)= \sum_{i\in S_p}
\one_{\{U_{p,i}<\sigma_i(\mathfrak{X}_{p,i})\}} h_{k,i}(\mathfrak{X}_{p,i}).
\end{equation}
Note that $\mathfrak{D}_p(k)$ is also what coordinate $k$ receives as exogenous input from $S_p$. That is, if we take
\begin{equation}
\label{eq:genarr}
\mathfrak{B}_{p,i}=\sum_{q\ne p}\mathfrak{D}_q(i),\ i \in S_p,
\end{equation}
we get another (more complex) representation of the dynamics of $F$ based on the point processes describing the interactions between the sets of the partition.
The $M$-RMF model associated with this partition features $M$ replicas
of this network with $q$ (vector state) nodes each.
In this $M$-RMF model, the {\em exogenous} output of node $i\in S_p$ of replica $m$ is
randomly sent to replicas chosen at random. More precisely, for all $q\ne p$,
the vector $(\mathfrak{D}_p^m(k),\ k\in S_q)$, with $\mathfrak{D}_p^m(k)$
defined as in (\ref{eq:gendep}), is sent to one replica chosen at random, and
this is done independently for all $q\ne p$.
This in turn defines new exogenous input point processes
$\mathfrak{B}_{p,i}^m$ as in (\ref{eq:genarr}).
Let $(\mathfrak{X}^m_{p,i},\ i\in S_p,\ p=1,\ldots,l,\ m=1,\ldots,M)$
denote the state variables in this $M$-RMF model.
It can be shown that, when $M$ tends to infinity,
\begin{enumerate}
\item for all $p\ne q$, the random state vectors
$(\mathfrak{X}^m_{p,i},\ i\in S_p)$ and $(\mathfrak{X}^m_{q,j},\ j\in S_q)$,
are asymptotically independent (although the coordinates of each vector are in general dependent);
\item for all $p$, and for all $m$, the exogenous arrivals $(\mathfrak{B}^m_{p,i},\ i\in S_p)$ to set $S_p$
tend to an independent multivariate compound Poisson variable
 with multivariate generating function
$$ \exp\left(\sum_{q\ne p} \sum_{n_i\in \mathbb N, i\in S_q} \sum_{s\subset S_q} \pi_{q,s,(n_i)} 
\left( 1-\prod_{i\in s} \prod_{k\in S_q} z_k^{h_{k,i}(n_i)}\right)\right).$$
In this last equation,
$$ \pi_{q,s,(n_i)}= 
\mathbf{P} [\tilde{\mathfrak{X}}_{q,i}=n_i,\ i\in S_q]
\prod_{j\in s} \sigma_j(n_j) \prod_{j'\in S_q\setminus s} (1-\sigma_{j'}(n_{j'})),$$ 
where $(\tilde{\mathfrak{X}}_{q,i})$ denotes random variables
with the limiting joint distribution assumed in the vector generalization of PAI.
\end{enumerate}
\section*{Conclusion}
A new class of discrete time dynamics involving point process based
interactions between interconnected nodes was introduced.
The Poisson Hypothesis was proved for the RMF version of such dynamics.
The proof is based on the property of pairwise asymptotic independence
between replicas and is by induction over time.
The key point is that randomized routing decisions on exchangeable events which are asymptotically
independent lead to Poisson point processes. As for future research, a natural question is whether these results extend to continuous-time and continuous-space versions of FIAP dynamics. 
The extension to continuous-time FIAPs will follow from similar arguments as presented here under the condition that the replica-limit/time-limit diagram commutes for FIAPs.
The extension to continuous-space FIAPs appears to require distinct analytical tools than those presented here.
Finally, another question of interest is whether the Poisson Hypothesis can be shown for the RMF limits of other classes of systems besides FIAPs.

\paragraph{Acknowledgements} The authors would like to thank N. Gast, S. Rybko, S. Shlosman and the anonymous referees for their comments on the present paper.

\paragraph{Funding}
T.T. was supported by the Alfred P. Sloan Research Fellowship FG-2017-9554.
F.B. was supported by an award from the Simons Foundation (\#197982). Both
awards are to the University of Texas at Austin. F.B. and M.D were supported the ERC NEMO grant (\# 788851) to INRIA Paris.

\bibliographystyle{plain}
\bibliography{discrete_case}

\end{document}